\documentclass[aps,prl,reprint,twocolumn]{revtex4-1}    

\usepackage[utf8]{inputenc}                  
\usepackage[T1]{fontenc}                     

\usepackage{amsthm,amsmath,amssymb} 
\usepackage{mathtools}
\usepackage{mathrsfs}                        
\usepackage{dsfont}                          
\usepackage{enumerate}               
\usepackage{hyperref}                        

\newcommand{\ud}{\mathrm{d}}

\newcommand{\CR}{\mathds{R}}

\newcommand{\depp}[2]{\textstyle\frac{\partial\emph{$#1$}}{\partial{\emph{$#2$}}}}

\newcommand{\escalar}[2]{\left\langle{\emph{$#1$}},{\emph{$#2$}}\right\rangle}

\newcommand{\rrightarrow}{\mathrel{\mathrlap{\rightarrow}\mkern1mu\rightarrow}}

\newenvironment{Cenumerate}{
\begin{enumerate}
  \setlength{\itemsep}{1pt}
  \setlength{\parskip}{0pt}
  \setlength{\parsep}{0pt}
}{\end{enumerate}}

\theoremstyle{plain}

\newtheorem{lemma}{Lemma}
\newtheorem{proposition}{Proposition}
\newtheorem{corollary}{Corollary}

\theoremstyle{definition}

\newtheorem{remark}{Remark}

\begin{document}


\title{Symplectic foliations induced by harmonic forms on 3-manifolds}

\author{Romero Solha}\affiliation{Department of Geometry and Topology, Faculty of Sciences, University of Granada. Postal address: Departamento de Geometría y Topología, Facultad de Ciencias - Universidad de Granada, Avenida Fuente Nueva S.N., Granada, Spain (postal code 18071). email: 
romerosolha@gmail.com.}\date{\today}
\begin{abstract}
This article details a construction of symplectic foliations on $3$-dimensional orientable riemannian manifolds from harmonic forms; and how it suggests a topological approach to Poisson's equation and newtonian gravity. 
\end{abstract}
\maketitle


\section{Nonvanishing harmonic 2-form}

On a $3$-dimensional orientable riemannian manifold $(M,\mathrm{g})$ the Hodge star operators $\star:\Omega^k(M)\to\Omega^{3-k}(M)$ satisfy $\star^2=\mathds{1}$ and take $2$-forms to $1$-forms. Using them, one can express the riemannian volume by $\star 1\in\Omega^3(M)$.  

\begin{lemma}\label{hodgenorm1form}
Given any $\alpha\in\Omega^2(M)$, and the unique solution $X\in\mathfrak{X}(M)$ to 
\begin{equation*}
\mathrm{g}(X,\boldsymbol{\cdot})=\star\alpha \ , 
\end{equation*}it holds
\begin{equation*}
\alpha\wedge\star\alpha=\mathrm{g}(X,X)\star1 
\end{equation*}and 
\begin{equation*}
\alpha=\imath_X\star1 \ .  
\end{equation*}
\end{lemma}
\begin{proof}
By the definition of the Hodge star operator applied to $1$-forms, for any $\beta\in\Omega^1(M)$, 
\begin{equation*}
\beta(X)\star1=\beta\wedge\star(\star\alpha)=\beta\wedge\alpha \ .
\end{equation*}Which implies  
\begin{equation*}
\mathrm{g}(X,X)\star1=(\star\alpha)\wedge\alpha=\alpha\wedge\star\alpha \ , 
\end{equation*}and, since $\beta\wedge\star1$ must be zero for it is a $4$-form,
\begin{equation*}
0=\imath_X(\beta\wedge\star1)=\beta(X)\star1-\beta\wedge\imath_X\star1=\beta\wedge(\alpha-\imath_X\star1) \ .
\end{equation*}
\end{proof}

If $\omega\in\Omega^2(M)$ is both closed and coclosed, 
\begin{equation*}
\left\{\begin{array}{c}
\ud\omega=0 \\
\star\ud\star\omega=0
\end{array}\right. \  , 
\end{equation*}and nonvanishing, then it provides a symplectic foliation on $M$. Unfortunately, even under this strong hypothesis ($\omega$ never vanishing pointwise), there is no consensus in the literature regarding the nomenclature of this geometric structure. Some authors call it cosymplectic, others might treat it as a symplectic vector bundle or a Poisson manifold. To avoid adhering to a nomenclature, the article limits itself to the description of  features of the underlying geometry when needed; the reader might want to consult \cite{GMP11}.

\begin{lemma}\label{geometrygravity}
Assuming $\omega\in\Omega^2(M)$ to be closed, coclosed, and nonvanishing, the kernel of $\star\omega$ defines a foliation in which $\omega$ restricts to a symplectic structure on each leaf. And together with the kernel of $\omega$,  
\begin{equation*}
\mathrm{ker}(\omega):=\mathrm{ker}\{\mathfrak{X}(M)\ni X\mapsto\imath_X\omega\in\Omega^1(M)\} \ , 
\end{equation*}the following orthogonal decomposition is true:
\begin{equation*}
\mathfrak{X}(M)=\mathrm{ker}(\omega)\oplus\mathrm{ker}(\star\omega) \ .
\end{equation*}
\end{lemma}
\begin{proof}
Considering the dual to $-\star\omega$ with respect to $\mathrm{g}$, the unique $R\in\mathfrak{X}(M)$ solving
\begin{equation*}
\mathrm{g}(R,\boldsymbol{\cdot})=-\star\omega \ , 
\end{equation*}lemma \ref{hodgenorm1form} together with the hypothesis of $\omega$ does not vanishing pointwise yield that the same holds for $\star\omega$, $R$, and $\omega\wedge\star\omega$. They also assure that the rank of the kernel of both forms are constant, $R\in\mathrm{ker}(\omega)$, and 
\begin{equation*}
\mathrm{ker}(\star\omega)=\{Y\in\mathfrak{X}(M) \ ; \ \mathrm{g}(R,Y)=0\} \ .
\end{equation*}

Because $\star\omega$ is closed, for any pair of vector fields $X,Y\in\mathrm{ker}(\star\omega)$,
\begin{align*}
0=\ud\star\omega(X,Y)&=X(\star\omega(Y))-Y(\star\omega(X))-\star\omega([X,Y]) \nonumber \\ 
&=-\star\omega([X,Y]) \ , 
\end{align*}showing that $\mathrm{ker}(\star\omega)$ is integrable and defines a foliation which the leaves are surfaces (because it is defined by the kernel of a nonvanishing $1$-form on a $3$-dimensional manifold). 

As for the symplectic nature of the leaves, assuming $X\in\mathrm{ker}(\star\omega)$ to be such that, for all $Y\in\mathrm{ker}(\star\omega)$, it holds $\omega(X,Y)=0$; thus
\begin{align*}
\imath_Y\circ\imath_X(\omega\wedge\star\omega)&=\imath_Y(( \imath_X\omega)\wedge\star\omega+\star\omega(X)\omega) \nonumber \\
&=\imath_Y((\imath_X\omega)\wedge\star\omega) \nonumber \\
&=\omega(X,Y)\star\omega-\star\omega(Y)\imath_X\omega=0 \ , 
\end{align*}making it impossible to exist $Z\in\mathfrak{X}(M)$ satisfying 
\begin{equation*}
\omega\wedge\star\omega(X,Y,Z)\neq 0 \ .
\end{equation*}Ergo, $X$ must be the zero vector field and $\omega|_{\mathrm{ker}(\star\omega)}$ is nondegenerate. Finally, the hypothesis of $\omega$ being closed secures that it restricts to a symplectic structure on each leaf of the foliation defined by the kernel of its Hodge dual. 

What is missing is the direct sum condition of the last claim. If $X\in\mathfrak{X}(M)$ belongs to the intersection $\mathrm{ker}(\omega)\cap\mathrm{ker}(\star\omega)$, the nondegeneracy of $\omega|_{\mathrm{ker}(\star\omega)}$ implies that $X=0$, as well as that the rank of $\mathrm{ker}(\omega)$ equals $1$ (because it is defined by the kernel of a nonvanishing $2$-form which is nondegenerate when restricted to $2$-dimensional submanifolds foliating a $3$-dimensional manifold).
\end{proof}


\section{Symplectic mapping torus}

An application of the previous lemma shows that a compact orientable $3$-dimensional manifold $M$ fibres over $S^1$ if and only if $S^1\times M$ has an $S^1$-invariant symplectic structure (which is related to Taubes's conjecture \cite{Taubesconjecture}), or it is a symplectic mapping torus; i.e. there exist a $2$-dimensional symplectic manifold $(\Sigma,\sigma)$ and a diffeomorphism $\psi:\Sigma\to\Sigma$ such that $\psi^*(\sigma)=\sigma$, and $M$ is $\Sigma\times [0,1]$ after identifying $(p,0)\in \Sigma\times [0,1]$ to $(\psi(p),1)\in\Sigma\times [0,1]$. 

Any generator of $H^1(S^1)$ can be pulled back to a nonvanishing closed $1$-form $\theta$, by the projection mapping of the fibration $M\rrightarrow S^1$. Since a nonvanishing $1$-form is transitive, Calabi \cite{Calabiharmonic} provides a riemannian structure such that $\star\theta$ is a closed $2$-form (it satisfies the hypothesis of lemma \ref{geometrygravity}); and results from \cite{Liarticle} justify the claim. 


\section{Area and mean curvature}  

Unless otherwise stated, the hypothesis of lemma \ref{geometrygravity} are assumed hereafter. 

\begin{proposition}\label{omegaareaform}
If $\Sigma\subset M$ is a leaf of $\mathrm{ker}(\star\omega)$, then $\omega$ restricted to it coincides with the induced area form from the restriction of $\mathrm{g}$ to $\Sigma$ multiplied by the norm of $R$ (which is equivalent to the norm of $\omega$). 
\end{proposition}
\begin{proof}
Lemma \ref{hodgenorm1form} provides $\omega\wedge\star\omega=\mathrm{g}(R,R)\star 1$ as well as $\omega=-\imath_R\star 1$; therefore, the normal  vector field to $\Sigma$ inducing the same orientation as $\omega$ is $-\mathrm{g}(R,R)^{-1/2}R$, whilst the induced area form from the restriction of $\mathrm{g}$ to $\Sigma$ is $-\mathrm{g}(R,R)^{-1/2}\imath_R\star 1$.
\end{proof}

\begin{corollary}\label{positivemass2}
Any compact (without boundary) and oriented leaf $\Sigma\subset M$ of $\mathrm{ker}(\star\omega)$ will satisfy $\int_\Sigma\omega>0$. 
\end{corollary}
\begin{proof}
Exact symplectic manifolds cannot be compact, $\Sigma$ is positively oriented with respect to the orientation induced by $\omega$, and the norm of $\omega$ is a positive function.   
\end{proof}

Because $-\mathrm{g}(R,R)^{-1/2}\imath_R\star 1$ is symplectic when restricted to $\Sigma$, and the mapping $\mathrm{j}$ defined by 
\begin{equation*}
\mathrm{g}\circ(\mathrm{j}\oplus\mathds{1})=-\frac{\imath_R\star 1}{\sqrt{\mathrm{g}(R,R)}}
\end{equation*}is a complex structure when restricted to the leaves of $\mathrm{ker}(\star\omega)$, one has a K\"{a}hler structure on $\Sigma$. 

As for the mean curvature of the leaves, they can be expressed in terms of $\omega$.

\begin{proposition}
The mean curvature of a leaf $\Sigma\subset M$ of $\mathrm{ker}(\star\omega)$ is $-\tfrac{1}{2}\star(\omega\wedge\ud(\star(\omega\wedge\star\omega))^{-1/2})$. 
\end{proposition}
\begin{proof}
By computing the mean curvature as the divergence of a normal vector field to the leaf, and using $\ud\omega=0$ and $\mathrm{g}(R,R)=\star(\omega\wedge\star\omega)$ (cf. lemma \ref{hodgenorm1form}), 
\begin{align*}
-\frac{1}{2}\mathrm{div}\left(\frac{-R}{\sqrt{\mathrm{g}(R,R)}}\right) 
&=-\frac{1}{2}\star\ud\left(\frac{-\imath_R\star 1}{\sqrt{\mathrm{g}(R,R)}}\right) \\
&=-\frac{1}{2}\star\ud\left(\frac{\omega}{\sqrt{\star(\omega\wedge\star\omega)}}\right) \\
&=-\frac{1}{2}\star((\ud(\star(\omega\wedge\star\omega))^{-1/2})\wedge\omega) \ . 
\end{align*}
\end{proof}



\section{Harmonic function}

When the Hodge dual $\star\omega$ is exact, the vector field dual to it is the gradient of a harmonic function.

\begin{lemma}\label{harmonicgradient}
If there exists a function $\phi\in C^\infty(M)$ such that $\star\omega=\ud\phi$, then $R\in\mathfrak{X}(M)$ solving 
\begin{equation*}
\mathrm{g}(R,\boldsymbol{\cdot})=-\star\omega 
\end{equation*}is the gradient of a harmonic function.
\end{lemma}
\begin{proof}
On the one hand,
\begin{equation*}
\mathrm{g}(R,\boldsymbol{\cdot})=-\star\omega=\ud(-\phi) \ . 
\end{equation*}On the other hand,
\begin{align*}
\Delta\phi&=(\star\ud\star)\circ\ud\phi=(\star\ud\star)(\star\omega) \nonumber \\
&=(\star\ud)\star^2\omega=\star\ud\omega=0 \ . 
\end{align*}
\end{proof}

As a consequence, level sets of harmonic functions inherit a symplectic structure. 

\begin{proposition}\label{gravitationalleaves}
The kernel of $\star\omega$ is the orthogonal complement of the distribution generated by the vector field $R$ (dual to $-\star\omega$ with respect to $\mathrm{g}$), and its symplectic leaves are the level sets of any harmonic function which the gradient is $R$.
\end{proposition}
\begin{proof}
From the previous definitions, 
\begin{align*}
\mathrm{ker}(\star\omega)&:=\{Y\in\mathfrak{X}(M) \ ; \ \star\omega(Y)=0\} \nonumber \\
&=\{Y\in\mathfrak{X}(M) \ ; \ \mathrm{g}(R,Y)=0\} \nonumber \\
&=\{Y\in\mathfrak{X}(M) \ ; \ \ud\phi(Y)=0\} \ .
\end{align*}
\end{proof}

This is a good opportunity to summarise the results. Starting with a $3$-dimensional orientable riemannian manifold $(M,\mathrm{g})$, any harmonic function $\phi\in C^\infty(M)$ has the property that its level sets are symplectic manifolds with respect to the $2$-form $\star\ud\phi$, cf. lemma \ref{geometrygravity} and proposition \ref{gravitationalleaves}; moreover, $\star\ud\phi$ induces an orientation on each level set, and its integral over the compact ones must be positive (cf. proposition \ref{positivemass2}). Those results are even more relevant when considering solutions of Poisson's equation, instead of Laplace's equation.


\section{Poisson's equation}

Let $\rho\in C^\infty(M)$ (or a distribution), then $\phi\in C^\infty(M)$ (which may not be smooth over the whole of $M$) satisfies the Poisson's equation with source $\rho$ whenever
\begin{equation*}
\Delta\phi=\rho \ . 
\end{equation*}By definition $\Delta\phi=\star\ud(\star\ud\phi)$; thus, $\rho\star1=\ud(\star\ud\phi)$ and Stokes's theorem gives, for any $3$-chain $V\subset M$,
\begin{equation*}
\int_V\rho\star1=\int_V\ud(\star\ud\phi)=\int_{\partial V}\star\ud\phi \ . 
\end{equation*}

Setting $J\subset M$ as the closure of the set where $\rho$ does not vanish (its support) and $M-J$ (the interior of $\rho^{-1}(\{0\})$), if $\Sigma\subset M-J$ is a $2$-cycle satisfying $\Sigma=\partial V$ on $M$, then 
\begin{equation*}
\int_V\rho\star 1=\int_{\Sigma}\star\ud\phi \ .
\end{equation*}Implying that the de Rham class $[\star\ud\phi]\in H^2(M-J)$ is predetermined by integrals of the source $\rho\in C^\infty(M)$ with respect to the volume $\star 1$.

This suggests another formulation to Poisson's equation. Instead of the initial data being a $3$-dimensional orientable riemannian manifold $(M,\mathrm{g})$ and a source function $\rho\in C^\infty(M)$ (or distribution), the information contained in $\rho$ is translated into a closed subset $J\subset M$ and an element of $H^2(M-J)$. And rather than looking for a function $\phi\in C^\infty(M)$ satisfying $\Delta\phi=\rho$, a solution would be a closed and coclosed form $\omega\in\Omega^2(M-J)$ such that $\star\omega$ is exact and $[\omega]\in H^2(M-J)$ coincides with the given de Rham class. The advantage of this formulation is apparent when $\rho$ is a distribution: for a solution $\phi$ of $\Delta\phi=\rho$ is not well defined over the support of $\rho$, whilst $\omega$ captures all pertinent information.

\begin{remark}
The subset where $(\star\ud\phi)\wedge(\ud\phi)$ fails to be a volume form needs not to be empty. Indeed, the $3$-form vanishes where the derivative of the harmonic function is zero; therefore, if $\star\ud\phi$ vanishes on an open subset, the harmonic function will be constant there, and $\star\ud\phi$ must be zero over the whole connected component containing it. Yet, lemma \ref{geometrygravity} and proposition \ref{positivemass2} are still valid for the points where $\omega\wedge\star\omega$ defines a volume form. 
\end{remark}


\section{Newtonian gravity}

Since newtonian gravity can be described in terms of a gravitational vector field determined by a function solving Poisson's equation, it is possible to exploit this alternative formulation to Poisson's equation to obtain an equivalent theory. 

Over a \textbf{spacelike region} described by a $3$-dimensional orientable riemannian manifold $(M,\mathrm{g})$ with gravitational sources located on a closed (and possibly not connected) subset $J\subset M$, the \textbf{gravitational field} is a closed differential $2$-form $\omega\in\Omega^2(M-J)$ such that, for every smooth singular 1-cycle $\gamma\subset M-J$, $\int_{\gamma}\star\omega=0$. The \textbf{gravitational mass} inside a $3$-dimensional spacelike region bounded by any gaussian surface $\Sigma$ is $\int_{\Sigma}\omega$. 

A \textbf{gaussian surface} $\Sigma\subset M-J$ is a $2$-dimensional compact, without boundary, and oriented submanifold such that $[\Sigma]\in H_2(M)$ is trivial. If $V\subset M$ is some $3$-dimensional submanifold satisfying $\partial V=\Sigma$, then the orientation of $\Sigma$ is the one induced by the orientation of the restriction of $\star 1$ to $V$ and used in Stokes theorem. 

\begin{remark}
The notion of a gaussian surface without a predetermined orientation renders it impossible to decide if the mass of a gravitational source is positive or negative: for a smooth singular $2$-cycle $\Sigma\subset M-J$ representing a generator $[\Sigma]\in H_2(M-J)$, the element $-[\Sigma]\in H_2(M-J)$ is also a generator and $\escalar{[\omega]}{-[\Sigma]}=-\int_{\Sigma}\omega$. 
\end{remark}

Given a \textbf{gravitational source distribution} $J\subset M$ and its respective gravitational mass, an element of $H^2(M-J)$ is fixed, and a representative $\omega$ will be the gravitational field for those sources whenever, for every smooth singular $1$-cycle $\gamma$, it holds $\int_{\gamma}\star\omega=0$.

Regarding dynamics, the \textbf{gravitational vector field} $R$ is the unique solution to $\mathrm{g}(R,\boldsymbol{\cdot})=-\star\omega$, using units where the gravitational constant equals $1/4\pi$. A test particle of gravitational mass $m$ under the influence of this gravitational field $\omega$ suffers a force given by $mR$. 


For illustration, taking $(M,\mathrm{g})$ to be the $3$-dimensional euclidean space and $
J=\{0\}\subset M=\CR^3$, using linear projections $x,y,z\in C^\infty(\CR^3)$ as coordinates
\begin{equation*}
\sigma:=\frac{1}{4\pi}\frac{z\ud x\wedge\ud y-y\ud x\wedge\ud z+x\ud y\wedge\ud z}{(x^2+y^2+z^2)^{\tfrac{3}{2}}}
\end{equation*}is coexact and its de Rham class generates $H^2(\CR^3-\{0\})$. Any element of $H^2(\CR^3-\{0\})$ associated to $m_0\in\CR$ through the isomorphism $H^2(\CR^3-\{0\})\cong\CR$ (induced by using $\{[\sigma]\}\subset H^2(\CR^3-\{0\})$ as a basis) can be represented by $\omega=m_0\sigma$, and if $\Sigma$ is any sphere centred at the origin oriented to be a gaussian surface, 
$\int_{\Sigma}\omega=m_0$. Supposing that $\omega$ is a gravitational field, its gravitational vector field, 
\begin{equation}\label{newtonsuniversallaw}
R=-\frac{m_0}{4\pi}\frac{x\depp{}{x}+y\depp{}{y}+z\depp{}{z}}{(x^2+y^2+z^2)^{\frac{3}{2}}} \ , 
\end{equation}coincides with the expected solution for the gravitational vector field generated by a particle of gravitational mass $m_0$ at the origin of the $3$-dimensional euclidean space; moreover, it is a multiple of the radial vector field, implying that the symplectic leaves of $\mathrm{ker}(\star\omega)$ are the spheres centred at the origin (cf. proposition \ref{gravitationalleaves}), i.e. $|m_0|=\int_{\Sigma}\omega>0$, because it is its symplectic area and the orientation of $\Sigma$ induced by $\omega$ will coincide with its original orientation only when $m_0>0$. 



\section{Assumptions and predictions}

Here are the assumptions made in the last section.

\begin{Cenumerate}
\item Physical space is described by a $3$-dimensional orientable riemannian manifold,
\item there exists a gravitational field described by a closed and coexact differential $2$-form,
\item the gravitational field is produced by sources (having gravitational mass) located on a closed subset of the physical space,
\item the gravitational mass ties to the differential $2$-form via integration over special subsets of the physical space (gaussian surfaces),
\item test particles interact with the source of a gravitational field through a vector field constructed out of the differential $2$-form and riemannian metric. 
\end{Cenumerate}

No sort of proof to under what conditions those assumptions are true, or if those conditions exist, is provide. Those are premises. However, they lead to predictions that can be tested and refuted. 

\begin{Cenumerate}
\item A test particle is attracted by a source of gravitational field with a force proportional to the product of their masses and inversely proportional to the square of the distance between their centres (acting along the line intersecting them both),
\item the gravitational field is linear with respect to gravitational sources, 
\item and Newton’s shell theorem is valid. 
\end{Cenumerate}

The first prediction follows from equation \ref{newtonsuniversallaw} (together with the fifth assumption), the second from the linearity of exterior derivatives, Hodge star operators, and relevant cohomology and homology groups, and the third from the topological invariance of cohomology and homology groups.   

Regarding the ability to test and refute those predictions, it is clear that the first is Newton's law of universal gravitation; hence, any experiment and observation that can test and refute newtonian gravity applies to the predictions of those new assumptions. In particular, said assumptions alone are not able to predict Mercury's orbit with better precision than Newton's theory. 

Whilst newtonian gravity has the first prediction as an assumption, the present article derives Newton's law of universal gravitation from a different set of assumptions. 

\begin{remark}
Although tempting, substituting gravitational by electric and mass by charge in the list of five assumptions will provide a limited theory describing Electrostatics phenomena. The reader must bear in mind that Electrostatics can be fully explained by Electromagnetism, a theory that must assume the physical space to be part of a set described by a $4$-dimensional lorentzian manifold; resulting in hyperbolic field equations ---drastically different from the 3-dimensional riemannian case.   
\end{remark}

\begin{remark}
Insisting in a parallel between Electrostatics and newtonian gravity, one might apply the assumptions of \cite{TheSolha} to gravity and reach the prediction that gravitational mass is an integral multiple of some fixed constant. Since there is evidence showing that this prediction is false, the logical conclusion is that Gravitation (being it relativistic or not) does not satisfy the assumptions of \cite{TheSolha}, and as a theory it must be different from Electromagnetism in its foundation.  
\end{remark}


\vspace{2em} 

\begin{acknowledgments}
This work was partially supported by PNPD/CAPES, the Basque government grant IT1094-16 of Spain, and Andalusia government grant P20\textunderscore01391 of Spain.
\end{acknowledgments}



\end{document}